\documentclass{article}

\usepackage{amssymb, amsmath, amsthm, array}

\usepackage[left=1.5in, top=1.3in, right=1.3in, bottom=1.2in]{geometry}

\newtheorem{theorem}{Theorem}[section]

\newtheorem{proposition}[theorem]{Proposition}
\newtheorem{corollary}[theorem]{Corollary}

\begin{document}

\title{\bf An Elementary Approach to Weight Multiplicities in Bivariate Irreducible Representations of Sp(2r)}
\author{Julia Maddox \\
		Department of Mathematics \\
		University of Oklahoma \\
		Norman, Oklahoma}
\date{}

\maketitle

\abstract{By bivariate irreducible representations of ${\rm Sp}(2r)$, we mean irreducible representations with highest weights containing at most two nonzero entries, using the usual identification of dominant weights for complex symplectic Lie algebras and their corresponding Lie groups as $r$-tuples in decreasing non-negative integers. This paper has two aims. The first aim is to provide a formula for the weight mulitplicities of said representations, which is easily computable. The second aim is to present these weight multiplicities using elementary means. The formula for these weight multiplicities is derived using basic multiliear algebra and combinatorial arguments through explicit descriptions of weight vectors.}

\vspace{0.2in}
classification numbers:  05A17, 05E10, 15A69, 17B10, 17B20, 22E60

\vspace{0.2in}
keywords: symplectic Lie algebra; symplectic Lie group; multilinear algebra; weight multiplicity; symmetric tensor; standard representation

\section{Introduction}

Every finite-dimensional representation of a complex semisimple Lie algebra can be completely decomposed into a direct sum of weight spaces.  The weights of a representation lie on a lattice in a Euclidean space whose dimension is equal to the rank of the Lie algebra.  Every irreducible representation has a unique highest weight and a unique weight diagram associated to it.  This unique weight diagram includes specific multiplicities for each weight; the multiplicity of a weight is equal to the dimension of its corresponding weight space.  These multiplicities appear in many contexts and have long been of interest to representation theorists.  Most formulas have been general and have incorporated complicated functions into summations, which are difficult to compute for specific representations and weights, especially in higher rank cases.

In this note, we will focus our efforts specifically on the Lie group ${\rm Sp}(2r)$ and its Lie algebra $\mathfrak{sp}(2r, \mathbb{C})$.  By narrowing our focus to bivariate irreducible representations of ${\rm Sp}(2r)$, we hope to garner more information about the weight vectors of these representations and formulate an expression for the weight multiplicities in more easily computable terms.

By bivariate irreducible representations of ${\rm Sp}(2r)$, we mean irreducible representations with highest weights containing at most two nonzero entries, using the usual identification of dominant weights for complex symplectic Lie algebras and their corresponding Lie groups as $r$-tuples in decreasing non-negative integers. This paper has two aims. The first aim is to provide a formula for the weight mulitplicities of said representations, which is easily computable. The second aim is to present these weight multiplicities using elementary means. The formula for these weight multiplicities is derived using basic multiliear algebra and combinatorial arguments through explicit descriptions of weight vectors.

We first prove a general result, which entails multilinear algebra for symmetric tensors; see Proposition \ref{prop1} and Corollary \ref{cor} from Section 2.  While these are certainly well known to experts, we have included direct proofs for completeness.  Proposition \ref{prop2} (and subsequently Corollary \ref{cor2}) follows from this result together with the explicit determination of certain highest weight vectors occurring in a tensor product of symmetric powers of the standard representation of $\mathfrak{sp}(2r, \mathbb{C})$ and a combinatorial argument.  Corollary \ref{cor2} then shows how an irreducible representation can be expressed as a linear combination of tensor products of symmetric powers of the standard representation.  These results can also be found using Littelmann's paper \cite{l} and Young tableaux or using a formula involving characters from Section 24.2 in \cite{fh}.

In Section 4, we first determine a formula for the weight multiplicities in a tensor product of symmetric powers of the standard representation using combinatorics.  We then use the identity of Section 3 and the weight multiplicities in a tensor product of symmetric powers of the standard representation to create a formula for the weight multiplicities in any bivariate irreducible representation of $\mathfrak{sp}(2r, \mathbb{C})$.

\section{Symmetric tensors} \label{mla}

Let $V$ have the basis $\{v_1, v_2, \ldots, v_k\}$, and let $W$ be a finite-dimensional vector space over the same field and with the basis $\{w_1, w_2, \ldots, w_l\}$.
Using the given bases of $V$ and $W$, the standard basis of ${\rm Sym}^{n}V \otimes {\rm Sym}^{m}W$ is $$\{ (v_{i_1} \cdots v_{i_n}) \otimes (w_{j_1} \cdots w_{j_m}) | i_1 \leq \ldots \leq i_n, j_1 \leq \ldots \leq j_m \}.$$
We may identify these basis elements with pairs of $k$- and $l$-tuples $(a_1, \ldots, a_k) \times (b_1, \ldots, b_l)$ using the following correspondence.  For a particular pure tensor basis element, let $a_i$ equal the number of times $v_i$ appears in that basis element and let $b_j$ equal the number of times $w_j$ appears in that basis element.  The standard basis of ${\rm Sym}^{n}V \otimes {\rm Sym}^{m}W$ is now identified as $$\{ (a_1, \ldots, a_k) \times (b_1, \ldots, b_l) | a_i, b_j \in \mathbb{Z}_{\geq 0}, \displaystyle\sum_{i=1}^k a_i = n, \displaystyle\sum_{j=1}^l b_j = m \}.$$

Let $V^*$ be the dual space of $V$ with dual basis $\{f_1, \ldots, f_k\}$.  Consider the linear map $$\rho: {\rm Sym}^{n-1}V \otimes {\rm Sym}^{m-1}V^* \to {\rm Sym}^{n}V \otimes {\rm Sym}^{m}V^*,$$ for $n,m \geq 1$, such that
$$(\alpha_1 \cdots \alpha_{n-1}) \otimes (\beta_1 \cdots \beta_{m-1}) \longmapsto \displaystyle\sum_{i=1}^k (\alpha_1 \cdots \alpha_{n-1} \cdot v_i) \otimes (\beta_1 \cdots \beta_{m-1} \cdot f_i).$$
This is the map defined as multiplication by the element $\displaystyle\sum_{i=1}^k v_i \otimes f_i$, which generates the trivial representation in $V \otimes V^*$, and $\rho$ is an injective intertwining map.

Dualizing the map $\rho$ (with $n$ and $m$ interchanged) produces the linear map $$\rho^*: {\rm Sym}^nV \otimes {\rm Sym}^mV^* \to {\rm Sym}^{n-1}V \otimes {\rm Sym}^{m-1}V^*$$ with the property
$$(\alpha_1 \cdots \alpha_n) \otimes (\beta_1 \cdots \beta_m) \longmapsto \displaystyle\sum_{i=1}^n \displaystyle\sum_{j=1}^m \beta_j(\alpha_i)(\alpha_1 \cdots \hat{\alpha_i} \cdots \alpha_n) \otimes (\beta_1 \cdots \hat{\beta_j} \cdots \beta_m),$$ 
and $\rho^*$ is a surjective intertwining map.

Since $\rho$ is an injective intertwining map, we recognize the following result.

\begin{proposition} \label{prop1}
Let $V$ be a finite-dimensional representation of a Lie algebra.  
Then there exists an invariant subspace 
\begin{center}
${\rm Sym}^{n-1}V \otimes {\rm Sym}^{m-1}V^* \subset {\rm Sym}^{n}V \otimes {\rm Sym}^{m}V^*$ for all integers $n,m\geq1$.
\end{center}
\end{proposition}  

\section{The case of $\mathfrak{sp}(2r, \mathbb{C})$} \label{sp2r}

We will apply the above result of Proposition \ref{prop1} to representations of the Lie algebra $\mathfrak{sp}(2r, \mathbb{C})$, $r \geq 2$, where

$\mathfrak{sp}(2r, \mathbb{C}) = \{ A \in \mathfrak{gl}(2r, \mathbb{C}) \mid A^tJ + JA = 0 \}$ and $J = \left[\begin{smallmatrix}
 0 & J_r \\
 -J_r & 0 \\
\end{smallmatrix}\right]$, in which $J_r$ is defined to be the $r \times r$ anti-diagonal matrix with ones along the anti-diagonal.

\noindent Evidently, $\mathfrak{sp}(2r, \mathbb{C})$ is $(2r^2+r$)-dimensional and has the following basis,

\begin{align*}
\{H_k\} &= \{E_{kk}-E_{2r+1-k,2r+1-k}|k=1, \ldots, r\},
\\
\{ X_{\alpha} \} &= \{ E_{ij} - E_{2r+1-j,2r+1-i}, E_{i,2r+1-j} + E_{j, 2r+1-i} | i < j\} \cup \{E_{i,2r+1-i} | i = 1, \ldots, r \},
\\
\{Y_{\alpha}\} &= \{X_{\alpha}^t\}.
\end{align*}

In this basis, the Cartan subalgebra is $\mathfrak{h} = \langle H_1, \ldots, H_r \rangle$, and for each root $\alpha$, $$\mathfrak{s}^{\alpha}=\text{span}\{X_{\alpha}, Y_{\alpha}, H_{\alpha} = [X_\alpha, Y_\alpha]\} \cong \mathfrak{sl}(2, \mathbb{C}).$$

Any weight $(x_1, x_2, \ldots,x_r)$ can be thought of as the eigenvalues associated to $H_1$ through $H_r$, respectively, for the corresponding weight vector.  The positive root associated to the root vector $E_{i,2r+1-i}$ is $(0, \ldots, 2, \ldots, 0)$ (the 2 is in the $i$th position), $i = 1, \ldots, r$.  The positive root associated to the root vector $E_{ij}-E_{2r+1-j,2r+1-i}$ is $(0, \ldots, 1, \ldots, -1, \ldots, 0)$ (with 1 in the $i$th position and -1 in the $j$th position), $1 \leq i < j \leq r$.  The positive root associated to the root vector $E_{i,2r+1-j}+E_{j,2r+1-i}$ is $(0, \ldots, 1, \ldots, 1, \ldots, 0)$ (with 1 in both the $i$th and $j$th position), $1 \leq i < j \leq r$.  Therefore the weight diagram of any representation will be symmetric about every coordinate plane $x_i = 0$, every plane of the form $x_i = x_j$, and every plane of the form $x_i = -x_j$ within $\mathbb{R}^r$.  As a result, any weight in the weight diagram of a representation, $(x_1, \ldots, x_r)$ will have the same multiplicity as the weight $(|x_{\sigma^{-1}(1)}|, \ldots, |x_{\sigma^{-1}(r)}|)$, where $\sigma \in S_r$.  The dominant Weyl chamber is $\{(x_1, \ldots, x_r) \in \mathbb{Z}^m :x_1 \geq x_2 \geq \ldots \geq x_r \geq 0\}$.  Let $V{(x_1, \ldots, x_r)}$ be the irreducible representation with highest weight $(x_1, \ldots, x_r)$.

The standard representation of $\mathfrak{sp}(2r, \mathbb{C})$ is $V{(1, 0, \ldots, 0)}$.  It has the standard basis of weight vectors, $\{e_1, \ldots, e_{2r} \}$, and is isomporphic to its dual representation with corresponding basis $\{f_1, \ldots, f_{2r} \}$.  These representations are isomorphic via $f_i \mapsto e_{2r+1-i}$ for $r+1 \leq i \leq 2r$ and $f_j \mapsto -e_{2r+1-j}$ for $1 \leq j \leq r$.

It can be easily shown that $V{(n,0, \ldots, 0)} = {\rm Sym}^nV{(1,0, \ldots, 0)}$.  The weight diagram for $V{(n,0, \ldots, 0)}$ is a series of nested "diamonds" with leading weights $(n-2k,0, \ldots, 0)$ and with multiplicities 
$\begin{pmatrix}
k + r-1
\\
r-1
\end{pmatrix}$ along the diamonds, $0 \leq k \leq \lfloor \frac{n}{2} \rfloor$.

\begin{proposition} \label{prop2}
For $\mathfrak{sp}(2r, \mathbb{C})$ and its standard representation $V=V{(1,0, \ldots, 0)}$, 
$${\rm Sym}^{n}V \otimes {\rm Sym}^{m}V = ({\rm Sym}^{n-1}V \otimes {\rm Sym}^{m-1}V) \oplus \bigoplus_{p=0}^{m} V{(n+m-p,p, 0, \ldots, 0)}$$ for integers $n \geq m \geq 1$.
\end{proposition}

\begin{proof}
Given $n \geq m$ and using the previously described notation, we define for all integers $p$ such that $0 \leq p \leq m$ the following vector in ${\rm Sym}^{n}V \otimes {\rm Sym}^{m}V^*$,
$$v_p = \displaystyle\sum_{i=0}^p \begin{pmatrix} p \\ i \end{pmatrix} (-1)^i (n-p+i, p-i,0, \ldots, 0) \times (0,\ldots, 0,i, m-i).$$
These vectors are in the kernel of the map $\rho^*$ defined in Section \ref{mla} because each $(n-p+i, p-i,0,\ldots,0) \times (0,\ldots,0,i, m-i) \mapsto 0 + \ldots + 0 = 0$.  Also, this vector is a highest weight vector with weight $(n+m-p,p, 0, \ldots, 0)$.

To see $v_p$ is a highest weight vector, it is enough to show that it is in the kernel of $X_{\alpha}$ for any $\alpha$.  

First, the only relevant calculations are $X_{\alpha}e_1$, $X_{\alpha}e_2$, $X_{\alpha}f_{2r-1}$, and $X_{\alpha}f_{2r}$.  These will all be equal to zero except for $X_{\alpha} = E_{12}-E_{2r-1,2r}$, and it is a straightforward calculation to show $(E_{12}-E_{2r-1,2r})v_p = 0$.

For each of the highest weight vectors, $v_p$, with weight $(n+m-p,p, 0, \ldots, 0)$ and in the kernel of $\rho^*$, there is an irreducible representation $V{(n+m-p,p, 0, \ldots, 0)}$ contained in the kernel.  Since all of the weights $\{(n+m-p,p, 0, \ldots, 0):0 \leq p \leq m\}$ are distinct, $$\bigoplus_{p=0}^{m} V{(n+m-p,p, 0, \ldots, 0)} \subset \ker(\rho^*).$$

It follows from semisimplicity and the surjectivity of $\rho^*$ that 
\begin{align*}
& ({\rm Sym}^{n-1}V \otimes {\rm Sym}^{m-1}V^*) \oplus \bigoplus_{p=0}^{m} V{(n+m-p,p)} 
\\
& \qquad \subset ({\rm Sym}^{n-1}V \otimes {\rm Sym}^{m-1}V^*) \oplus \ker(\rho^*)
\\
& \qquad = {\rm Sym}^{n}V \otimes {\rm Sym}^{m}V^*
\\
& \qquad = {\rm Sym}^{n}V \otimes {\rm Sym}^{m}V
\end{align*}
for $n \geq m \geq 1$.
Note that $V^*$ can be replaced by $V$ since this representation is self-dual.

We now wish to prove this inclusion is an equality.  Let $v$ be some highest weight vector in the kernel of $\rho^*$.  We will show $v$ is a scalar multiple of some $v_p$, and thus $\ker(\rho^*) = \bigoplus_{p=0}^{m} V{(n+m-p,p, 0, \ldots, 0)}$.

The standard basis of ${\rm Sym}^{n}V \otimes {\rm Sym}^{m}V^*$ is a basis of weight vectors.  As a weight vector, $v$ exists in some weight space of ${\rm Sym}^{n}V \otimes {\rm Sym}^{m}V^*$ and must be equal to a linear combination of basis vectors of the form $(a_1, \ldots, a_{2r}) \times (b_1, \ldots, b_{2r})$ such that each of these basis vectors have the same weight.  Since $v$ is a highest weight vector, $X_{\alpha}.v = 0$ for $X_{\alpha}$.  We will now use the following calculations.

For the first set of $X_{\alpha}$, with $1 \leq i < j \leq r$,
\begin{align*}
&(E_{ij} - E_{2r+1-j,2r+1-i}).(a_1, \ldots, a_{2r}) \times (b_1, \ldots, b_{2r})
\\
&= a_j(a_1, \ldots, a_{i}+1, \ldots, a_{j}-1, \ldots, a_{2r}) \times (b_1, \ldots, b_{2r})
\\
&\,\, - a_{2r+1-i}(a_1, \ldots, a_{2r+1-j}+1, \ldots, a_{2r+1-i}-1, \ldots, a_{2r}) \times (b_1, \ldots, b_{2r})
\\
&\,\, - b_i(a_1, \ldots, a_{2r}) \times (b_1, \ldots, b_i -1, \ldots, b_j +1, \ldots, b_{2r})
\\
&\,\, + b_{2r+1-j}(a_1, \ldots, a_{2r}) \times (b_1, \ldots, b_{2r+1-j}-1, \ldots, b_{2r+1-i}+1, \ldots, b_{2r}).
\end{align*}

For the second set of $X_{\alpha}$, with $1 \leq i < j \leq r$,
\begin{align*}
&(E_{i,2r+1-j} + E_{j,2r+1-i}).(a_1, \ldots, a_{2r}) \times (b_1, \ldots, b_{2r})
\\
&= a_{2r+1-j}(a_1, \ldots, a_i + 1, \ldots, a_{2r+1-j}-1, \ldots, a_{2r}) \times (b_1, \ldots, b_{2r})
\\
&\,\,+a_{2r+1-i}(a_1, \ldots, a_j+1, \ldots, a_{2r+1-i}-1, \ldots, a_{2r}) \times (b_1, \ldots, b_{2r})
\\
&\,\,-b_i(a_1, \ldots, a_{2r}) \times (b_1, \ldots, b_i-1, \ldots, b_{2r+1-j}+1, \ldots, b_{2r})
\\
&\,\,-b_j(a_1, \ldots, a_{2r}) \times (b_1, \ldots, b_j-1, \ldots, b_{2r+1-i}+1, \ldots, b_{2r}).
\end{align*}

For the third set of $X_{\alpha}$, with $i = 1, \ldots, r$,
\begin{align*}
&(E_{i,2r+1-i}).(a_1, \ldots, a_{2r}) \times (b_1, \ldots, b_{2r})
\\
&= a_{2r+1-i}(a_1, \ldots, a_i+1, \ldots, a_{2r+1-i}-1, \ldots, a_{2r}) \times (b_1, \ldots, b_{2r})
\\
&\,\,-b_i(a_1, \ldots, a_{2r}) \times (b_1, \ldots, b_i -1, \ldots, b_{2r+1-i}+1, \ldots, b_{2r}).
\end{align*}

Note that there is a slight abuse of notation here as in the final identity, if $a_{2r+1-i} = 0$, there is simply no first term after the equals sign and $(a_1, \ldots, a_i+1, \ldots, a_{2r+1-i}-1, \ldots, a_{2r}) \times (b_1, \ldots, b_{2r})$ is not a vector.

Out of all the basis vectors $(a_1, \ldots, a_{2r}) \times (b_1, \ldots, a_{2r})$ with nonzero coefficients in the linear combination of $v$, let $(a_1, \ldots, a_{2r}) \times (b_1, \ldots, b_{2r})$ be a specific basis vector such that $a_1$ is maximal.  Without loss of generality, assume the coefficient of this basis vector in the linear combination of $v$ is equal to $1$.

In the expansion of $(E_{1,j}-E_{2r+1-j,2r}).v$ for $j=2, \ldots, r$, if $a_j \neq 0$ the vector $a_j(a_1+1, \ldots, a_j - 1, \ldots, a_{2r}) \times (b_1, \ldots, b_{2r})$ will occur when $(E_{1,j}-E_{2r+1-j,2r})$ is applied to $(a_1, \ldots, a_{2r}) \times (b_1, \ldots, b_{2r})$.  Since $a_1$ was maximal, no other basis vector with a nonzero coefficient in the linear combination of $v$ will produce the basis element $(a_1+1, \ldots, a_j - 1, \ldots, a_{2r}) \times (b_1, \ldots, b_{2r})$ in the image when $(E_{1,j}-E_{2r+1-j,2r})$ is applied to $v$, but if this is a defined basis element, its coefficient in $(E_{1,j}-E_{2r+1-j,2r}).v = 0$ must equal zero.  This is a contradiction.  Therefore, $a_j = 0$ for $j=2, \ldots, r$.  Similarly, $(E_{1,2r+1-j}+E_{j,2r}).v = 0$ implies $a_{2r+1-j} = 0$ for $j=2, \ldots, r$ or $a_j = 0$ for $j= r+1, \ldots, 2r-1$, and $(E_{1,2r}).v = 0$ implies $a_{2r}= 0$.  Now, $(a_1, \ldots, a_{2r}) \times (b_1, \ldots, b_{2r}) = (a_1, 0, \ldots, 0) \times (b_1, \ldots, b_{2r})$, and $a_1$ must be equal to $n$, which means $a_1 = n \geq 1$.  

Consider $(E_{i,2r+1-i}).v = 0$ for $i=2, \ldots, r$.  If $b_i \neq 0$, then $(E_{i,2r+1-i}).(a_1, 0, \ldots, 0) \times (b_1, \ldots, b_{2r}) = -b_i(a_1, 0, \ldots, 0)\times(b_1, \ldots, b_i-1, \ldots, b_{2r+1-i}+1, \ldots, b_{2r})$ and the maximality of $a_1$ again implies that there will be no other vector to cancel with this resultant vector in the expansion of $(E_{i,2r+1-i}).v = 0$, which is a contradiction, and thus $b_i = 0$ for $i=2, \ldots, r$.  Now, $(a_1, \ldots, a_{2r}) \times (b_1, \ldots, b_{2r}) = (a_1, 0, \ldots, 0) \times (b_1, 0, \ldots, 0, b_{r+1}, \ldots, b_{2r})$.

Consider $(E_{ij}-E_{2r+1-j,2r+1-i}).v = 0$ for $2 \leq i < j \leq r$.  If $b_{2r+1-j} \neq 0$, then \begin{align*}
&(E_{ij}-E_{2r+1-j,2r+1-i}).(a_1, 0, \ldots, 0) \times (b_1, 0, \ldots, 0, b_{r+1}, \ldots, b_{2r}) 
\\
&\,\, = b_{2r+1-j}(a_1, 0, \ldots, 0) \times (0, \ldots, 0, b_{r+1}, \ldots, b_{2r+1-j}-1, \ldots, b_{2r+1-i}+1, \ldots, b_{2r})
\end{align*}
and $(a_1, 0, \ldots, 0) \times (b_1, 0, \ldots, 0, b_{r+1}, \ldots, b_{2r})$ is the only basis vector to produce this resultant basis vector in the expansion of $(E_{ij}-E_{2r+1-j,2r+1-i}).v = 0$.  Since the coefficient of this basis vector must be equal to zero, there is a contradiction, and thus $b_{2r+1-j} = 0$ for $j = 3, \ldots, 2r$.  Now, $(a_1, \ldots, a_{2r}) \times (b_1, \ldots, b_{2r}) = (a_1, 0, \ldots, 0) \times (b_1, 0, \ldots, 0, b_{2r-1}, b_{2r})$.  Note that whenever writing $v$ as a linear combination of standard basis vectors, any basis vector with a nonzero coefficient and maximal $a_1$ must have this form.

Assume $b_1 \neq 0$.  Consider $(E_{1,j}-E_{2r+1-j,2r}).v$ for $j=2, \ldots, r$.  The only basis vectors to produce $(a_1, 0, \ldots, 0) \times (b_1-1, 0, \ldots, 1, \ldots, 0, b_{2r-1}, b_{2r})$ (with $1$ in the $j$th position in the second $2r$-tuple) in this expansion are $(a_1, 0, \ldots, 0) \times (b_1, 0, \ldots, 0, b_{2r-1}, b_{2r})$ and $(a_1-1, 0, \ldots, 1, \ldots, 0) \times (b_1-1, 0, \ldots, 1, \ldots, 0, b_{2r-1}, b_{2r})$ (with $1$ in the $j$th position in both $2r$-tuples).  $(E_{1,j}-E_{2r+1-j,2r}).v=0$ necessitates that the coefficient of $(a_1-1, 0, \ldots, 1, \ldots, 0) \times (b_1-1, 0, \ldots, 1, \ldots, 0, b_{2r-1}, b_{2r})$ (with $1$ in the $j$th position in both $2r$-tuples) in the linear combination of $v$ must be equal to $b_1$.  Similarly, $(E_{1,2r+1-j}+E_{j,2r}).v = 0$ for $j=3, \ldots, r$ implies the coefficient of $(a_1-1, 0, \ldots, 1, \ldots, 0) \times (b_1 - 1, 0, \ldots, 1, \ldots, 0, b_{2r-1}, b_{2r})$ (with $1$ in the $2r+1-j$th position in both $2r$-tuples) in the linear combination of $v$ must be equal to $b_1$.  $(E_{1,2r+1-j}+E_{j,2r}).v = 0$ implies the coefficient of $(a_1-1, 0, \ldots, 0, 1, 0) \times (b_1 - 1, 0, \ldots, 0, b_{2r-1}+1, b_{2r})$ in the linear combination of $v$ must be equal to $b_1$.  $(E_{1,2r}).v = 0$ implies the coefficient of $(a_1-1, 0, \ldots, 0, 1) \times (b_1 - 1, 0, \ldots, 0, b_{2r-1}, b_{2r}+1)$ in the linear combination of $v$ must be equal to $b_1$.

Since $v$ is in $\ker(\rho^*)$, $\rho^*(v) = 0$, and this implies either $b_1 = 0$ and $$\rho^*((a_1, 0, \ldots, 0) \times (b_1, 0, \ldots,0,b_{2r-1},b_{2r})) = 0$$ or 
\begin{align*}
&\rho^*((a_1, 0, \ldots, 0) \times (b_1, 0, \ldots,0,b_{2r-1},b_{2r})) \\
&\qquad = a_1b_1(a_1-1, 0,\ldots, 0) \times (b_1-1, 0, \ldots, 0,b_{2r-1}, b_{2r})
\end{align*}
will cancel with the image of other basis vectors in the expansion of $\rho^*(v)$.  Consider the following identity.
\begin{align*}
&\rho^*((a_1, \ldots, a_{2r}) \times (b_1, \ldots, b_{2r}))
\\
&=\displaystyle\sum_{i=1}^{2r} a_ib_i(a_1, \ldots, a_i-1, \ldots, a_{2r})\times(b_1, \ldots, b_i-1, \ldots, b_{2r}),
\end{align*}
with the same slight abuse of notation as before.

The only other vectors, which will produce $(a_1-1, 0, \ldots, 0) \times (b_1-1, 0, \ldots, 0, b_{2r-1}, b_{2r})$ in the image of $\rho^*$, are $(a_1-1, 0, \ldots, 1, \ldots, 0) \times (b_1 - 1, 0, \ldots, 1, \ldots, 0, b_{2r-1}, b_{2r})$ (with $1$ in the $i$th position in both $2r$-tuples) for $i=2, \ldots, 2r-2$, $(a_1-1, 0, \ldots, 0, 1, 0) \times (b_1 - 1, 0, \ldots, 0, b_{2r-1}+1, b_{2r})$, and $(a_1-1, 0, \ldots, 0, 1) \times (b_1 - 1, 0, \ldots, 0, b_{2r-1}, b_{2r}+1)$.  Provided $b_1 \neq 0$, we already know what the coefficients of these vectors must be in $v$.

In the image of $\rho^*$, the coefficient of $(a_1-1, 0, \ldots, 0) \times (b_1-1, 0, \ldots, 0, b_{2r-1}, b_{2r})$ is equal to 
\begin{align*}
&a_1b_1  + b_1(1) + \ldots + b_1(1) + b_1(b_{2r-1}+1) + b_1(b_{2r}+1)
\\
&=b_1(a_1 + 2r-1 + b_{2r-1} + b_{2r})
\\
&=0.
\end{align*}
If $b_1 \neq 0$, $a_1 + 2r-1 + b_{2r-1} + b_{2r} = 0$, but this is a contradiction because $a_1$, $b_{2r-1}$, and $b_{2r}$ are all non-negative integers.  Therefore $b_1 = 0$.

When $v$ is written as a linear combination of standard basis vectors, the weight vector $(a_1, 0, \ldots, 0) \times (0, \ldots, 0, b_{2r-1}, b_{2r})$ has a nonzero coefficient.  Therefore $v$ has the same weight as this vector.  Another way to write $(a_1, 0, \ldots, 0) \times (0, \ldots, 0, b_{2r-1}, b_{2r})$ is as $(n, 0, \ldots, 0) \times (0, \ldots, 0, j, m-j)$ for some integer $j$ such that $0 \leq j \leq m$.  

The multiplicity of the weight $(n+m-j, j, 0, \ldots, 0)$ in ${\rm Sym}^{n}V \otimes {\rm Sym}^{m}V^* = {\rm Sym}^{n}V \otimes {\rm Sym}^{m}V$ is equal to the number of integer solutions to the equation $x_1 + x_2 = m$ such that $0 \leq x_1 \leq n+m-j$ and $0 \leq x_2 \leq j$.  (See the next section for an explanation of this.)  The only such solutions are $(m,0), \ldots, (m-j,j)$.  Therefore the multiplicity is $j+1$.

The weight space for $(n+m-j,j, 0, \ldots, 0)$ is spanned by the linearly independent vectors $Y_{\alpha}^jv_0, Y_{\alpha}^{j-1}v_1, \ldots, Y_{\alpha}^{0}v_j$ for $Y_{\alpha} = E_{12}^t - E_{2r-1,2r}^t$.  This completely accounts for the multiplicity of the weight $(n+m-j,j, 0, \ldots, 0)$.  Therefore $v$ must be a scalar multiple of $v_j$.
\end{proof}

All of the highest weight vectors in ${\rm Sym}^nV \otimes {\rm Sym}^mV$, $V=V(1, 0, \ldots, 0)$, can be determined using the proof of Proposition \ref{prop2}, the map $\rho$ from Section \ref{mla}, and the isomorphism between the standard representation and its dual.

In \cite{l}, Littelmann provides a generalization of the Littlewood-Richardson rule in the cases of all simple, simply connected algebraic groups of type $A_m$, $B_m$, $C_m$, $D_m$, $G_2$, $E_6$, and partial results for $F_4$, $E_7$, and $E_8$.  This generalization provides an algorithm for decomposing tensor products of irreducible representations using Young tableaux and can be utilized to produce the result of Proposition \ref{prop2}.

\begin{corollary} \label{cor2}
Let $V = V(1, 0, \ldots, 0)$.  For integers $n \geq m = 1$, 
\begin{align*} {\rm Sym}^n V \otimes V = {\rm Sym}^{n+1}V \oplus V{(n,1, 0, \ldots, 0)} \oplus {\rm Sym}^{n-1}V
\end{align*}
For $n \geq m \geq 2$, 
\begin{align*} 
& ({\rm Sym}^nV \otimes {\rm Sym}^mV) \oplus ({\rm Sym}^{n}V \otimes {\rm Sym}^{m-2}V)
\\
& \hspace{0.3in} = ({\rm Sym}^{n+1}V \otimes {\rm Sym}^{m-1}V) \oplus V{(n,m, 0, \ldots, 0)} \oplus ({\rm Sym}^{n-1}V \otimes {\rm Sym}^{m-1}V)
\end{align*}
\end{corollary}

\begin{proof}
The first assertion is the special case of Proposition \ref{prop2} where $m=1$.
Using Proposition \ref{prop2}, when $n \geq m \geq 2$, 
\begin{align*}
& {\rm Sym}^nV \otimes {\rm Sym}^mV = ({\rm Sym}^{n-1}V \otimes {\rm Sym}^{m-1}V) \oplus \bigoplus_{p=0}^{m} V{(n+m-p,p,0, \ldots, 0)} 
\end{align*} 
and
\begin{align*}
& {\rm Sym}^{n+1}V \otimes {\rm Sym}^{m-1}V = ({\rm Sym}^{n}V \otimes {\rm Sym}^{m-2}V) \oplus \bigoplus_{p=0}^{m-1} V{(n+m-p,p, 0, \ldots, 0)} .
\end{align*}
Combining these equations yields the assertion.
\end{proof}

In the Grothendieck group of all representations of $\mathfrak{sp}(2r, \mathbb{C})$, for $V = V(1,0, \ldots, 0)$, where subtraction is equivalent to cancellation, we get

\begin{align*}
V(n, 0, \ldots 0) &= {\rm Sym}^nV & n \geq 0
\\
V(n,1, 0, \ldots, 0) &= {\rm Sym}^nV \otimes V - {\rm Sym}^{n+1}V - {\rm Sym}^{n-1}V & n \geq 1
\\
V(n,m, 0, \ldots, 0) &= {\rm Sym}^{n}V \otimes {\rm Sym}^{m}V + {\rm Sym}^{n}V \otimes {\rm Sym}^{m-2}V & n \geq m \geq 2
\\
& \hspace{0.4in} - {\rm Sym}^{n-1}V \otimes {\rm Sym}^{m-1}V - {\rm Sym}^{n+1}V \otimes {\rm Sym}^{m-1}V.
\end{align*}

This result can also be derived using character theory as in a proposition in Section 24.2 in \cite{fh}, which gives a formula for the character of an irreducible representation of a simplectic Lie algebra in terms of the characters of symmetric powers of the standard representation.

These identities also descend to the weight spaces of the weights in the respective representations because the weight space of a reducible representation can be thought of as the direct sum of the weight spaces for that same weight in its irreducible components.

\section{Weight multiplicities in bivariate irreducible representations}

As proven above, any bivariate irreducible representation of $\mathfrak{sp}(2r, \mathbb{C})$ with highest weight $(n,m, 0, \ldots, 0)$ can be written as a formal combination of tensor products of symmetric powers of the standard representation, and in fact all finite-dimensional representations of $\mathfrak{sp}(2r, \mathbb{C})$ can be written as formal combinations of tensor products of symmetric powers of the standard representation.  Therefore by determining the weight multiplicities in ${\rm Sym}^nV \otimes {\rm Sym}^mV$ for the standard representation $V$ and any $n$ and $m$, we will be able to deduce the weight multiplicities in $V(n,m, 0, \ldots, 0)$.  We will now use a combinatorial argument to produce a formula for the weight multiplicities in ${\rm Sym}^nV \otimes {\rm Sym}^mV$.

For this combinatorial argument, we will employ Proposition \ref{propsieve}, and we define $\begin{pmatrix} a \\ n \end{pmatrix}$ for any integer $a$ and any non-negative integer $n$ as follows,
\begin{align*}
\begin{pmatrix} a \\ n \end{pmatrix} &= \left\{ \begin{array}{lr} 0 & a < n \\ \frac{a!}{n!(a-n)!} & a \geq n. \end{array} \right. 
\end{align*}

\begin{proposition} \label{propsieve}
Given integer $n \geq 1$, non-negative integers $c_1, \ldots, c_n$, and any integer $k$, the number of solutions to the equation $x_1 + \ldots + x_n = m$ in non-negative integers, such that $0 \leq x_j \leq c_j$ for any $j$, is $f(c_1, \ldots, c_n)(m)$, and $$f(c_1, \ldots, c_n)(m) = \displaystyle\sum_{i=0}^n (-1)^i \displaystyle\sum_{s \in S_i} \begin{pmatrix} m - (s+i) + n-1 \\ n-1 \end{pmatrix},$$
where $S_i$ is the multiset consisting of all subsums of $c_1 + \ldots + c_n$ with length $i$.
\end{proposition}

Here $S_0 = \{ 0 \}$.
\begin{proof}
This can be proven by Inclusion-Exclusion or a sieve method.  For this method, consider some subset of $\{c_j\}$, without loss of generality assume it is $\{c_{1}, \ldots, c_{l}\}$.  The number of non-negative integer solutions to the equation $x_1 + \ldots + x_n = m$ such that $x_{j} > c_{j}$ for $1 \leq j \leq l$ is equal to the number of non-negative integer solutions of (2) from the following equations,
\begin{align}
(x_1 - (c_1 + 1)) + \ldots + (x_l - (c_l+1)) + x_{l+1} + \ldots + x_n  &= m - (c_1 + 1 + \ldots c_l + 1)\\
y_1 + \ldots + y_l + x_{l+1} + \ldots + x_n &= m - (c_1 + \ldots + c_l + l).
\end{align}
The number of non-negative integer solutions to (2) is $\begin{pmatrix} m - (s+l) + n-1 \\ n-1 \end{pmatrix}$, where $s = c_1 + \ldots + c_l$.
\end{proof}

To understand the weights and weight vectors of ${\rm Sym}^nV \otimes {\rm Sym}^mV$, we will first describe the weights and weight vectors for ${\rm Sym}^nV$.  Any standard basis vector of ${\rm Sym}^nV$ $(a_1, \ldots, a_{2r})$ has weight $(a_1 - a_{2r}, \ldots, a_{r}-a_{r+1}) = (x_1, \ldots, x_r)$, and 
\begin{align*}
a_1 &= x_1 + a_{2r} \\
&\vdots \\
a_r &= x_r + a_{r+1}.
\end{align*}
Also, any weight $(x_1, \ldots, x_r)$ of ${\rm Sym}^nV$ has some standard basis vector $(a_1, \ldots, a_{2r})$ satisfying these equations.  Given some weight $(x_1, \ldots, x_r)$ of ${\rm Sym}^nV$ with corresponding weight vector $(a_1, \ldots, a_{2r})$,
\begin{align*}
&a_1 + \ldots + a_{r} + a_{r+1} + \ldots + a_{2r} \\
&= (x_1 + a_{2r}) + \ldots + (x_r + a_{r+1}) + a_{r+1} + \ldots + a_{2r} \\
&= (x_1 + \ldots + x_r) + 2(a_{r+1} + \ldots + a_{2r}) \\
&= n.
\end{align*}
Therefore $n - (x_1 + \ldots + x_r) = 2k$ for some integer $k$.  By the symmetries of the weight diagrams of $\mathfrak{sp}(2r, \mathbb{C})$, we may assume $0 \leq x_i \leq n$ and $0 \leq k \leq \lfloor \frac{n}{2} \rfloor$.  For $(x_1, \ldots, x_r)$ in non-negative integers such that $n - (x_1 + \ldots + x_r) = 2k$, $0 \leq k \leq \lfloor \frac{n}{2} \rfloor$, any vector $(a_1, \ldots, a_{2r})$ such that $a_{r+1} + \ldots + a_{2r} = k$ and $a_i = x_i + a_{2r+1-i}$ for $i = 1, \ldots, r$ will have weight $(x_1, \ldots, x_r)$ and any vector with this weight must have this form, so the multiplicity of $(x_1, \ldots, x_r)$ is equal to the number of non-negative integer solutions to $a_{r+1} + \ldots + a_{2r} = k$, which is $\begin{pmatrix}
k + r-1
\\
r-1
\end{pmatrix}$.

Any solution in non-negative integers to $x_1 + \ldots + x_r = n-2k$ for each $0 \leq k \leq \lfloor \frac{n}{2} \rfloor$ will be a weight of ${\rm Sym}^nV$ because there exists a vector $(a_1, \ldots, a_{2r})$ of ${\rm Sym}^nV$ such that $a_{r+1} + \ldots + a_{2r} = k$ and $a_i = x_i + a_{2r+1-i}$ for $i = 1, \ldots, r$.  This proves that the weight diagram of ${\rm Sym}^nV$ will consist of weights for each $0 \leq k \leq \lfloor \frac{n}{2} \rfloor$, which are the integer solutions to $|x_1| + \ldots + |x_r| = n - 2k$.  (These are the "diamonds" with leading weights $(n-2k, 0, \ldots, 0)$.)  By the symmetries of the weight diagrams, every weight satisfying $|x_1| + \ldots + |x_r| = n - 2k$ for a particular $k$ will have multiplicity $\begin{pmatrix}
k + r-1
\\
r-1
\end{pmatrix}$.

Using the previous notation, the set $$\{ (a_1, \ldots, a_{2r}) \times (b_1, \ldots, b_{2r}) | a_i, b_j \in \mathbb{Z}_{\geq 0}, \displaystyle\sum_{i=1}^{2r} a_i = n, \displaystyle\sum_{j=1}^{2r} b_j = m \}$$ is a basis of weight vectors for the representation ${\rm Sym}^nV \otimes {\rm Sym}^mV$.  The weight of $(a_1, \ldots, a_{2r}) \times (b_1, \ldots, b_{2r})$ is 
\begin{align*}
&((a_1+b_1) - (a_{2r}+b_{2r}), \ldots, (a_r + b_r) - (a_{r+1}+b_{r+1})) 
\\
&\qquad = (c_1 - c_{2r}, \ldots, c_r - c_{r+1}),
\end{align*}
by defining $c_i = a_i + b_i$.  With a similar reasoning as above by viewing $(c_1, \ldots, c_{2r})$ as a vector of ${\rm Sym}^{n+m}V$, the weights of ${\rm Sym}^nV \otimes {\rm Sym}^mV$ will be all integer solutions to $|x_1| + \ldots + |x_{r}| = n + m - 2k$ for $0 \leq k \leq \lfloor \frac{n+m}{2} \rfloor$.  The question is now what is the multiplicity of a particular weight $(x_1, \ldots, x_{r})$ of ${\rm Sym}^nV \otimes {\rm Sym}^mV$.

\begin{theorem} \label{thm}
Given $(x_1, \ldots, x_r) \in \mathbb{Z}^r$ with $\displaystyle\sum_{j=1}^r |x_j| = n + m - 2k$ for some $0 \leq k \leq \lfloor \frac{n+m}{2} \rfloor$, the multiplicity of $(x_1, \ldots, x_r)$ in ${\rm Sym}^nV \otimes {\rm Sym}^mV$ is equal to
$$\displaystyle\sum_{\substack{c_j \in \mathbb{Z}_{\geq 0} \\ c_{r+1} + \ldots + c_{2r}=k}} \, \, \displaystyle\sum_{i=0}^{2r} (-1)^i \displaystyle\sum_{s \in S_i} \begin{pmatrix} m - (s+i) + 2r-1 \\ 2r-1 \end{pmatrix},$$
where $S_i$ is the multiset consisting of all subsums of $(c_1 + \ldots + c_{2r})$ with length $i$ and $c_j = |x_j| + c_{2r+1-j}$ for $1 \leq j \leq r$.
\end{theorem}

\begin{proof}
Without loss of generality, assume $(x_1, \ldots, x_r) = (|x_1|, \ldots, |x_r|)$.  The $2r$-tuples in non-negative integers, $(c_1, \ldots, c_{2r})$, such that $(c_1 - c_{2r}, \ldots, c_r - c_{r+1}) = (x_1, \ldots, x_{r})$ is equal to the set of such $2r$-tuples with $c_{r+1} + \ldots + c_{2r} = k$ and $c_i = x_i + c_{2r+1-i}$ for $1 \leq i \leq r$.  For each of these $2r$-tuples, the number of ways to write $(a_1, \ldots, a_{2r}) \times (b_1, \ldots, b_{2r})$ in non-negative integers such that $\displaystyle\sum_{i=1}^{2r} a_i = n$, $\displaystyle\sum_{i=1}^{2r} b_i = m$, and $c_i = a_i + b_i$ is $f(c_1, \ldots, c_{2r})(m)$, as defined in Proposition \ref{propsieve}.  Therefore the multiplicity of $(x_1, \ldots, x_{r})$ in ${\rm Sym}^nV \otimes {\rm Sym}^mV$ is $$\displaystyle\sum_{\substack{c_j \in \mathbb{Z}_{\geq 0} \\ c_{r+1} + \ldots + c_{2r}=k}} f(c_1, \ldots, c_{2r})(m)$$ such that $c_i = x_i + c_{2r+1-i}$ for $1 \leq i \leq r$.
\end{proof}

The benefit of the formula in Theorem \ref{thm} is that the summations used are easily understood and the summands are simple binomial coefficients.  This leads to a more easily computable formula for weight multiplicities.  Now we will consider the weights of $V(n,m,0, \ldots, 0)$.  The weights of $V(n,m, 0, \ldots, 0)$ are contained within ${\rm Sym}^nV \otimes {\rm Sym}^mV$ and must be of the form $(x_1, \ldots, x_r)$ for integers $x_i$ such that $\displaystyle\sum_{i=1}^r |x_i| = n + m - 2k$ for some $0 \leq k \leq \lfloor \frac{n+m}{2} \rfloor$.

\begin{theorem}
Given $(x_1, \ldots, x_r) \in \mathbb{Z}^r$ with $\displaystyle\sum_{j=1}^r |x_j| = n + m - 2k$ for some $0 \leq k \leq \lfloor \frac{n+m}{2} \rfloor$, the multiplicity of $(x_1, \ldots, x_r)$ in $V(n,m, 0, \ldots, 0)$ is equal to
\begin{align*}
&\displaystyle\sum_{\substack{c_j \in \mathbb{Z}_{\geq 0} \\ c_{r+1} + \ldots + c_{2r}=k}} \, \, \displaystyle\sum_{i=0}^{2r} (-1)^i \displaystyle\sum_{s \in S_i} \begin{pmatrix} m - (s+i) + 2r-2 \\ 2r-2 \end{pmatrix}
\\
&-\displaystyle\sum_{\substack{c_j' \in \mathbb{Z}_{\geq 0} \\ c_{r+1}' + \ldots + c_{2r}'=k-1}} \, \, \displaystyle\sum_{i=0}^{2r} (-1)^i \displaystyle\sum_{s \in S_i'} \begin{pmatrix} m - 1 - (s+i) + 2r-2 \\ 2r-2 \end{pmatrix},
\end{align*}
where $S_i$ is the multiset consisting of all subsums of $(c_1 + \ldots + c_{2r})$ with length $i$ and $c_j = |x_j| + c_{2r+1-j}$ for $1 \leq j \leq r$ and $S_i'$ is the multiset consisting of all subsums of $(c_1' + \ldots + c_{2r}')$ with length $i$ and $c_j' = |x_j| + c_{2r+1-j}'$ for $1 \leq j \leq r$.
\end{theorem}

\begin{proof}
As shown in Section \ref{sp2r}, 
\begin{align*}
V(n,m, 0, \ldots, 0) &= {\rm Sym}^{n}V \otimes {\rm Sym}^{m}V - {\rm Sym}^{n+1}V \otimes {\rm Sym}^{m-1}V
\\
& \hspace{0.4in} + {\rm Sym}^{n}V \otimes {\rm Sym}^{m-2}V - {\rm Sym}^{n-1}V \otimes {\rm Sym}^{m-1}V. 
\end{align*}
This identity is true for any integers $n \geq m \geq 0$ with the understanding that ${\rm Sym}^iV$ is equal to $0$ for negative values of $i$.  Combining this identity and Theorem \ref{thm} proves the assertion after applying the fact that
\begin{align*}
n+m-2k &= (n+1) + (m-1) -2k
\\
&= n + (m-2) - 2(k-1)
\\
&= (n-1) + (m-1) - 2(k-1),
\end{align*}
and the binomial coefficient identity,
$\begin{pmatrix} n+1 \\ k \end{pmatrix} 
- \begin{pmatrix} n \\ k \end{pmatrix} 
= \begin{pmatrix} n \\ k-1 \end{pmatrix}$.
\end{proof}

\end{document}